\documentclass[a4paper,11pt]{amsart}
\usepackage{amsfonts,amsmath,amssymb,amsthm,mathtools}
\usepackage{geometry} 
\usepackage{fixmath} 
\usepackage{enumitem}

\theoremstyle{plain}
\newtheorem{theorem}{Theorem}[section]

\newtheorem{proposition}[theorem]{Proposition}
\theoremstyle{remark}
\newtheorem{remark}[theorem]{Remark}

\theoremstyle{definition}

\newtheorem{assumptionletter}{Assumptions}

\newcommand{\N}{\mathbb{N}}
\newcommand{\Z}{\mathbb{Z}}
\newcommand{\R}{\mathbb{R}}

\DeclareMathOperator{\Var}{Var}

\begin{document}

\title[On the speed of ERWRE]{On the speed of a one-dimensional random walk in a random environment perturbed by cookies of strength one}

\author{Elisabeth Bauernschubert}

\date{January 16, 2015}

\begin{abstract}
  We consider a random walk in an i.i.d.\ random environment on $\Z$ that is perturbed by cookies of strength 1. The number of cookies per site is assumed to be i.i.d.\ Results on the speed of the random walk are obtained.\\
  Our main tool is the correspondence in certain cases between the random walk and a branching process in a random environment with migration.
\end{abstract}

\maketitle

\section{Introduction}

In \cite{Bauernschubert12} and \cite{Bauernschubert13} 
the author studied a left-transient (respectively recurrent) one-dimensional random walk in a random environment that is
perturbed by cookies of maximal strength and established criteria for transience and recurrence. In the current article, we study the speed of this random walk. 

We recall the model from \cite{Bauernschubert12, Bauernschubert13} and explain it in a few words. Choose
a sequence $(p_x)_{x\in\Z}$, with $p_x\in (0,1)$ for all $x\in\Z$, at random and put on every integer $x\in\Z$ a random number $M_x$ of cookies ($M_x\in\N_0$).  
Now, start a nearest-neighbor random walk at $0$: If the walker encounters a cookie on his current position $x$, he consumes it and is excited to jump to $x+1$ a.s. If there is no cookie, he goes to $x+1$ with probability $p_x$ and to $x-1$ with probability $1-p_x$. For illustrations of the model see \cite[Fig.~1]{Bauernschubert12} or \cite[Figure~1]{Bauernschubert13}.

For a precise definition, denote by $\Omega:=([0,1]^{\N})^{\Z}$ the space of so-called \emph{environments}.
Let $(\Omega', \mathcal{F}')$ be a suitable measurable space with probability measures $P_{x,\omega}$ for $\omega\in\Omega$ and $x\in\Z$ and $(S_n)_{n\geq0}$ a process on $\Omega'$, such that for all $\omega=((\omega(x,i))_{i\geq 1})_{x\in\Z}$ and all $z\in\Z$
\begin{align*}
P_{z,\omega}[S_0=z]&=1,\\
P_{z,\omega}[S_{n+1}=S_n+1|(S_m)_{1\leq m\leq n}]&= \omega(S_n,\#\{m\leq n:\: S_m=S_n\}),\\
P_{z,\omega}[S_{n+1}=S_n-1|(S_m)_{1\leq m\leq n}]&= 1-\omega(S_n,\#\{m\leq n:\: S_m=S_n\})
\end{align*}
is satisfied.
The so-called \emph{excited random walk} (ERW for short) $(S_n)_{n\geq0}$ is a nearest-neighbor random walk under $P_{z,\omega}$ that starts in $z$ and whose transition probability upon the $i$th visit to site $y\in\Z$ is given by $\omega(y,i)$.
In the usual notion, $\omega(y,i)$ is also said to be the strength of the $i$th cookie on site $y$.

The elements $\omega\in\Omega$ itself are chosen at random according to some probability measure $\mathbb{P}$ on $\Omega$. Averaging the so-called \emph{quenched} measure $P_{x,\omega}$ over the environments $\omega$ yields the \emph{annealed} or \emph{averaged} measure $P_x[\cdot]:=\mathbb{E}[P_{x,\omega}[\cdot]]$ on the product space $\Omega\times\Omega'$. By $\mathbb{E}$, $E_{x,\omega}$ and $E_x$, we denote the expectation operators respectively.

The discussion of excited random walks started with \cite{Benjamini} where a simple symmetric random walk (in $\Z^d$, $d\geq1$) is disturbed by one cookie at each site. The model, which is also known as cookie random walk, has been generalized in various ways, e.g.\ in the one-dimensional case among others by Zerner \cite{Zerner05},  Basdevant and Singh \cite{Basdevant08b} and Kosygina and Zerner \cite{Zerner08}.
For a recent survey on ERWs see \cite{KosyginaZerner12}.

In our setting, we consider cookies of strength 1. For each $x\in\Z$, the number of cookies of maximal strength at site $x$ is defined by
\begin{align*}
&M_x:=\sup\{i\geq 1:  \;\omega(x,j)=1\quad\forall 1\leq j\leq i\}
\end{align*}
with the convention $\sup\emptyset =0$.

Throughout the paper, all or parts of the following assumptions on $\mathbb{P}$ will be needed, compare also to \cite{Bauernschubert12, Bauernschubert13}:

\begin{assumptionletter}
         \label{as_A_paper3}
         \renewcommand{\theenumi}{A.\arabic{enumi}}
         \renewcommand{\labelenumi}{\theenumi}
         There is $\mathbb{P}$-a.s.\ a sequence $(p_x)_{x\in\Z}\in (0,1)^{\Z}$ such that the following holds.
         \begin{enumerate}
                 \item \label{as_A1_paper3} $\omega(x,i)=p_x$ for all $x\in\Z$ and for all $i> M_x$.
                 \item \label{as_A2_paper3} $(p_x,M_x)_{x\in\Z}$ is i.i.d.\ and $\{p_x, M_x; x\in\Z\}$ is independent under $\mathbb{P}$.
		 \item \label{as_A3_paper3} $\mathbb{E}[|\log \rho_0|^2]<\infty$ and $\mathbb{E}[\rho_0^2]<\infty$ where $\rho_x:=({1-p_x})p_x^{-1}$ for $x\in\Z$.
	     \item \label{as_A4_paper3} $\mathbb{P}[\rho_0=1]<1$.
	     \item \label{as_A5_paper3} $\mathbb{P}[M_0=0]>0$.
        \end{enumerate}
\end{assumptionletter} 

If Assumption \ref{as_A_paper3} holds with $\mathbb{P}[M_x=0]=1$, the process belongs to the class of \emph{random walks in random environments} (RWRE for short). For an overview and results on RWREs we refer the reader to \cite{Zeitouni} and references therein.
In the most studied ERW model, a simple symmetric random walk is perturbed by cookies; commonly, the number of cookies per site is bounded, but the cookies may have strength between 0 and~1, see e.g.\ \cite{Basdevant08b, Zerner08, KosyginaZerner12}.
In order to emphasize that the underlying process in our model is an RWRE,  
we call our model described above \emph{excited random walk in a random environment} (ERWRE for short). This model has already been introduced by the author in \cite{Bauernschubert12, Bauernschubert13}.
Assumption \ref{as_A4_paper3} excludes the simple symmetric random walk as underlying dynamic. By Assumption \ref{as_A5_paper3} we avoid the trivial case where the random walker encounters at least one cookie on every integer $\mathbb{P}$-a.s.

Under Assumption \ref{as_A1_paper3}, $\omega$ is given $\mathbb{P}$-a.s.\ by a sequence $(p,M):=(p_x,M_x)_{x\in\Z}$.
For clarity and convenience let us therefore write $P_{x,(p,M)}$ for the quenched measure instead of $P_{x,\omega}$ and just $P_{(p,M)}$ if $x=0$.

For a random walk in an i.i.d.\ random environment, Solomon proved in \cite[Theorem~(1.7)]{Solomon75} the following recurrence and transience criteria.
Suppose that Assumptions \ref{as_A1_paper3} and \ref{as_A2_paper3} hold with $\mathbb{P}[M_x=0]=1$ and that $\mathbb{E}[\log \rho_0]$ is well defined. Then, $P_0$-a.s.,
$\lim_{n\to\infty}S_n = +\infty$ if $\mathbb{E}[\log \rho_0]<0$,
$\lim_{n\to\infty}S_n = -\infty$ if $\mathbb{E}[\log \rho_0]>0$ and
$(S_n)_{n\geq0}$ is recurrent
if $\mathbb{E}[\log \rho_0]=0$. 

Theorem 1.1 in \cite{Bauernschubert12} and Theorem 1 in \cite{Bauernschubert13} provide transience and recurrence criteria for an ERWRE with underlying left-transient or recurrent RWRE.
In the present version we dropped the restriction $\mathbb{P}[M_x=\infty]=0$ by allowing $\mathbb{P}[M_x=\infty]\in[0,1)$.
The notation $(\cdot)_+$ abbreviates $\max(0,\cdot)$.

\begin{theorem}[\cite{Bauernschubert12}]
  \label{th:Bauernschubert12}
 Let Assumption \ref{as_A_paper3} hold and assume that
 $\mathbb{E}[\log\rho_0]>0$.
\begin{itemize}[font=\normalfont]
 \item[(i)] If $\mathbb{E}[(\log M_0)_+]<\infty$, then
$\lim_{n\to\infty}S_n=-\infty$ $P_0$-a.s.
 \item[(ii)]  If $\mathbb{E}[(\log M_0)_+]=\infty$ and if
$\limsup_{t\to\infty}(t\cdot\mathbb{P}[\log{M_0}>t])<\mathbb{E}[\log \rho_0]$,
then $S_n=0$ infinitely often $P_0$-a.s.
 \item[(iii)] If
$\liminf_{t\to\infty}(t\cdot\mathbb{P}[\log{M_0}>t])>\mathbb{E}[\log \rho_0]$,
then $\lim_{n\to\infty}S_n=+\infty$ $P_0$-a.s. 
\end{itemize}
\end{theorem}

Similar criteria in the case where the RWRE is recurrent are provided in \cite{Bauernschubert13}.

In the current work we study how ``fast`` the random walks in Theorem \ref{th:Bauernschubert12} and in \cite{Bauernschubert13} go to
infinity when they are transient (to the left or to the right).
Therefore note that by \cite[Theorem~4.1]{KosyginaZerner12} $(S_n)_{n\geq0}$
satisfies in the setting of Theorem \ref{th:Bauernschubert12} a strong law of large numbers, i.e.\ there 
exists a non-random $\nu\in[-1,1]$ such that
  \[
    \lim_{n\to\infty}\frac{S_n}{n}=\nu \quad P_0\text{-a.s.}
  \]
  
  This limit $\nu$ is called \emph{speed} or \emph{velocity} of the random walk,
  see also \cite[Sections~4 and~5]{KosyginaZerner12}. If the underlying dynamic in the ERW model is the simple symmetric random walk and if the number of cookies (here with strength strictly in between 0 and 1) 
  is bounded --- i.e.\ there is a deterministic $K\in\N$ such that $\omega(x,i)=\frac{1}{2}$ for all $i>K$ and $x\in\Z$ a.s.\ ---,
  then results on the speed can be found in \cite[Theorem~1.1]{Basdevant08b} and \cite[Theorem~2]{Zerner08}. The key parameter in this case turned out to be the average total drift per site
  \[
    \bar{\delta}:=\mathbb{E}\Bigg[\sum_{i\geq1}{(2\omega(0,i)-1)}\Bigg].
  \]
  Under some weak ellipticity assumptions, it has been obtained that $\nu=0$ if $\bar{\delta}\in[-2,2]$, $\nu<0$ if $\bar{\delta}<-2$ and $\nu>0$ if $\bar{\delta}>2$, \cite{Basdevant08b, Zerner08}.

The speed for an RWRE is given in \cite[Theorem (1.16)]{Solomon75}, see also \cite[Theorem 2.1.9 + Remark]{Zeitouni}.
Let us briefly recall Solomon's result.
\begin{theorem}[\cite{Solomon75}]
  \label{th:RWRE_speed}
  Let Assumptions \ref{as_A1_paper3} and \ref{as_A2_paper3} hold with $\mathbb{P}[M_x=0]=1$ (RWRE) and let $\mathbb{E}[\log\rho_0]$ be well defined. Then, $P_0$-a.s.,
   \begin{itemize}[font=\normalfont]
    \item[(i)] $\nu=\frac{1-\mathbb{E}[\rho_0]}{1+\mathbb{E}[\rho_0]}>0$ if $\mathbb{E}[\rho_0]<1$,
    \item[(ii)] $\nu=-\frac{1-\mathbb{E}[\rho_0^{-1}]}{1+\mathbb{E}[\rho_0^{-1}]}<0$ if $\mathbb{E}[\rho_0^{-1}]<1$ and
    \item[(iii)] $\nu=0$ if $\mathbb{E}[\rho_0]^{-1}\leq 1\leq\mathbb{E}[\rho_0^{-1}]$.
   \end{itemize}
\end{theorem}

For the ERWRE model studied in \cite{Bauernschubert12, Bauernschubert13}  
we will show the following results on the speed in the present paper.
\begin{theorem}
  \label{th:1_transientRE}
  Let Assumption $\ref{as_A_paper3}$ hold and suppose the underlying RWRE is left-transient or recurrent, i.e.\ $\mathbb{E}[\log\rho_0]\geq0$.
     \begin{itemize}[font=\normalfont]
      \item[(i)] If $\mathbb{E}[M_0]<\infty$ and $\mathbb{E}[\rho_0^{-1}]<1$ then the ERWRE goes to $-\infty$ with negative speed: $\nu=\lim_{n\to\infty}\frac{S_n}{n}<0$ $P_0$-a.s.
      \item[(ii)] If $\mathbb{E}[M_0]=\infty$ or $\mathbb{E}[\rho_0^{-1}]\geq1$, and if additionally $\mathbb{E}[\rho_0] > \mathbb{P}[M_0<\infty]^{-1}$ then
      $\nu=0$.
      \item[(iii)] If $\mathbb{E}[\rho_0] < \mathbb{P}[M_0<\infty]^{-1}$ then $\nu>0$. 
     \end{itemize}
  \end{theorem}
Note that (i)-(iii) in Theorem \ref{th:1_transientRE} cover all cases except for $\mathbb{E}[\rho_0] = \mathbb{P}[M_0<\infty]^{-1}$.

\begin{remark}
  By Jensen's inequality and Assumption \ref{as_A4_paper3}, $\mathbb{E}[\log\rho_0]\geq0$ implies that $\mathbb{E}[\rho_0]>1$. Hence, if $M_0$ is finite $\mathbb{P}$-a.s.\ and if the underlying RWRE is left-transient or recurrent there is no chance for the random walker to go to infinity with positive speed in the setting of Theorem \ref{th:1_transientRE}. This is only possible if there are ``enough'' infinite cookie stacks.
  On the other hand, if the RWRE tends to $-\infty$ with negative speed (i.e.\ $\mathbb{E}[\rho_0^{-1}]<1$), the cookies may slow it down without changing its transience behavior. According to Theorems \ref{th:Bauernschubert12} and \ref{th:1_transientRE} this occurs if $\mathbb{E}[M_0]=\infty$, but $\mathbb{E}[(\log M_0)_+]<\infty$.
\end{remark}  
 
Further questions concern an ERWRE where the underlying dynamic is transient to the right with zero speed. Can cookies accelerate this RWRE? How many cookies of maximal strength can be placed without increasing the speed and what is the influence of the distribution of $\rho_0$? Answers --- but not yet in a complete version --- are given in the next theorem.  
\begin{theorem}
  \label{th:2_transientRE}
  Let Assumption $\ref{as_A_paper3}$ hold and suppose that the underlying RWRE is right-transient with zero speed, i.e.\ $\mathbb{E}[\log\rho_0]< 0$ and $\mathbb{E}[\rho_0]\geq 1$.
  \begin{itemize}[font=\normalfont]
   \item[(i)] Assume that $\gamma\,\mathbb{E}[\rho_0]\,\mathbb{P}[M_0<\infty]>1$, where $\gamma=\mathbb{E}[\rho_0^{\beta}]$ with $\beta$ such that \linebreak
   $\mathbb{E}[\rho_0^{\beta}\log\rho_0]=0$. Then $\nu=\lim_{n\to\infty}\frac{S_n}{n}=0$, $P_0$-a.s.
   \item[(ii)] If $\mathbb{E}[\rho_0]\geq\mathbb{P}[M_0=0]^{-1}$ then $\nu=0$.
   \item[(iii)] If $\mathbb{E}[\rho_0] < \mathbb{P}[M_0<\infty]^{-1}$ then $\nu>0$.
  \end{itemize}
  \end{theorem}

\begin{remark}
  Let us remark that $\beta$ in Theorem \ref{th:2_transientRE}(i) exists, is unique and $0<\beta<1$. Moreover $\gamma <1$.
  To see this use the moment generating function $g(t):=\mathbb{E}[\rho_0^{t}]$, $t\in\R$, and recall its properties e.g.\ from \cite{Billingsley_probability}.
  By Assumption \ref{as_A3_paper3}, $g(t)$ is finite on $[0,2]$.
  Furthermore note that the derivative is $g'(t)=\mathbb{E}[\rho_0^{t}\log\rho_0]$ and under the assumptions of Theorem \ref{th:2_transientRE},
  $g(0)=1$, $1\leq g(1)<\infty$ and $-\infty<g'(0)<0$.
\end{remark}
  
Note that it is still open if it is possible to obtain positive speed if the underlying RWRE is transient to the right with zero speed and $M_0$ is finite $\mathbb{P}$-a.s.
  
To prove Theorems \ref{th:1_transientRE} and \ref{th:2_transientRE} we use three tools.
In the situation when the ERWRE goes to $-\infty$ it is not hard to prove non-zero or zero speed.
Basically one uses the formulation of the speed known from e.g.\ \cite[Theorem 13]{Zerner05} or \cite[Section~7]{Zerner08}.
One method to study the speed when the ERWRE goes to infinity, is based on a well-known regeneration or renewal structure of these random walks and a relation to certain branching processes, see e.g.\ \cite{KestenKozlovSpitzer75} in the case of RWREs and \cite[Section~6]{Zerner08}, \cite[Sections~4 and 5]{KosyginaZerner12} and references therein in the case of ERWs. 
The right-transient ERWRE will be related to a specific \emph{branching process in a random environment with migration} (BPMRE for short).
Its velocity is then positive or non-positive according to whether the expected total size of the BPMRE until its first time of extinction is finite or infinite.
Therefore, this work also contains a result on BPMRE.
As a third tool --- especially used when we deal with infinite cookie stacks and in order to obtain positive speed --- we simply apply results about RWREs and exploit the monotonicity of $\nu$ with respect to the cookie environment, see \cite[Proposition 4.2]{KosyginaZerner12}.

The article is organized as follows. The next section is devoted to the study of a specific branching process in a random environment with emigration. It is slightly different to the BPMRE that corresponds to the ERWRE in the case of transience to the right, but is later used to prove that the expected total size of the latter branching process up to its first time of dying out is infinite. In Section~\ref{sec:translationERWRE_BPMRE} the correspondence between $(S_n)_{n\geq0}$ and a BPMRE is given.
Section~\ref{sec:proof_of_speed} finally contains the proofs of Theorems~\ref{th:1_transientRE} and \ref{th:2_transientRE}. 

\section{Branching process in a random environment with migration}
\label{sec:supercriticalBPMRE}

In this section we introduce a branching process in a random environment with emigration. It will be similar to the BPMRE that is related to the ERWRE in Section~\ref{sec:translationERWRE_BPMRE}.
The first process has the advantage of being easier to handle. For convenience and in view of its application to $(S_n)_{n\geq0}$, let us define the branching process on $\Omega'$.
Therefore, we assume without loss of generality that there is a family $\{\xi_{i}^{(n)}; i,n\in\N\}$ of independent random variables on $\Omega'$ such that, for $\mathbb{P}$-a.e.\ $(p,M)$,  
\[
 P_{(p,M)}[\xi_{i}^{(n)}=k]=(1-p_n)^k\cdot p_n \quad\text{for}\quad k\in\N_0,
\]
 i.e.\ for all $i\in\N$, $\xi_{i}^{(n)}$ is geometrically distributed with parameter $p_n$.
Let us define now $Z_0:=1$ and for $n\geq1$ recursively
  \[Z_n:=\left(\sum_{i=1}^{Z_{n-1}}\xi_{i}^{(n)}-M_n\right)_+.\] 

This process belongs to the class of branching processes in an i.i.d.\ random environment with migration.
In our setting there is no immigration and the number of emigrants is unbounded.
Furthermore, in the above definition, the number of emigrants is immediately subtracted from the population size.
Note that 0 is an absorbing state for $(Z_n)_{n\geq0}$.

Given an environment $(p,M)$, the expected number of offspring in generation $n$ per individual is 
\[
 E_{(p,M)}[\xi_{1}^{(n)}]=\frac{1-p_n}{p_n}=\rho_n
\]
and its variance is 
\[
 \Var_{(p,M)}[\xi_{1}^{(n)}]=\frac{1-p_n}{p_n^2}.
\]

The literature on branching processes in general is extensive, see for instance \cite{Athreya, Haccou_Jagers_Vatutin2005, Vatutin93}.
If there is no migration in any generation, i.e.\ $M_n=0$ for all $n\in\N$, then $(Z_n)_{n\geq0}$ belongs to the class of \emph{branching processes in random environments} (BPRE for short), see for instance \cite{Smith69, Athreya} or \cite{BirknerGeigerKersting05} and references therein.
The branching process combining the concept of reproduction according to a random environment with the phenomenon of migration --- and here especially unbounded emigration --- does not seem to be broadly discussed. To the best of our knowledge, the results given in Proposition \ref{prop:supercriticalBPMRE} are not yet covered by the literature. Therefore, we will prove Proposition \ref{prop:supercriticalBPMRE}, that is required for our study of ERWREs, directly in this section.

We use the usual classification of BPREs, see e.g.\ \cite{Athreya} or \cite{BirknerGeigerKersting05}, and call $(Z_n)_{n\geq0}$ \emph{subcritical, critical} or \emph{supercritical} according to whether $\mathbb{E}[\log\rho_1] <0, =0$ or $> 0$.
The BPRE dies out a.s.\ in the subcritical and critical regime, whereas --- under a certain integrability condition --- the supercritical BPRE may explode, see \cite{Smith69, Athreya}.
Note that the process $(Z_n)_{n\geq0}$ is heuristically similar to some random difference equation: $Z_n$ should be more or less $(\rho_n\cdot Z_{n-1}-M_n)_+$. This similarity helps to study the BPRE with emigration in the proof of the following proposition. For some heuristic to Proposition~\ref{prop:supercriticalBPMRE} we refer the reader to the Remarks~\ref{remark:heuristic} and \ref{remark:heuristic2} below.

\begin{proposition}
  \label{prop:supercriticalBPMRE}
  Let Assumption \ref{as_A_paper3} hold and assume $\mathbb{E}[\rho_1] > \mathbb{P}[M_1<\infty]^{-1}$.
  The total population size of $(Z_n)_{n\geq0}$ has infinite mean, i.e.\ $E_0[\sum_{j\geq 0} Z_j]=\infty$,
  if one of the following conditions holds.
  \begin{itemize}[font=\normalfont]
    \item[(i)] $(Z_n)_{n\geq0}$ is supercritical or critical $(\mathbb{E}[\log\rho_1]\geq 0)$.
    \item[(ii)] $(Z_n)_{n\geq0}$ is subcritical $(\mathbb{E}[\log\rho_1]<0)$ and
  $\gamma\mathbb{E}[\rho_1]\mathbb{P}[M_1<\infty]>1$, where $\gamma=\mathbb{E}[\rho_1^{\beta}]$ with $\beta$ such that $\mathbb{E}[\rho_1^{\beta}\log\rho_1]=0$.
  \end{itemize}
\end{proposition}

\begin{remark}
  \label{remark:heuristic}
  Let us give a short heuristic for Proposition \ref{prop:supercriticalBPMRE} in the supercritical setting. Consider a sequence that grows exponentially until some ``catastrophe'' happens
  that causes extinction.
  Precisely, for some $a>1$ let $\tilde{X}_0:=1$ and recursively $\tilde{X}_n:=a \tilde{X}_{n-1}$ if $M_n<a \tilde{X}_{n-1}$ and $\tilde{X}_n:=0$ otherwise, for $n\in\N$.
  Now, calculations show that for every $m\in\N$
    \[
      \mathbb{E}\Bigg[\sum_{j\geq 0} \tilde{X}_j\Bigg]\geq \mathbb{E}[\tilde{X}_m|T_0^{\tilde{X}}>m]\cdot\mathbb{P}[T_0^{\tilde{X}}>m]= a^m\cdot \prod_{k=1}^{m}\mathbb{P}[M_1<a^k] 
    \]
  where $T_0^{\tilde{X}}:=\inf\{n\geq1: \tilde{X}_n=0\}$. Thus, the expected sum of $\tilde{X}_j$ is infinite if $a\cdot\mathbb{P}[M_1<\infty]>1$.
\end{remark}

\begin{remark}
  \label{remark:heuristic2}
  Note that Proposition \ref{prop:supercriticalBPMRE} covers supercritical, critical and some subcritical BPREs with emigration.
  The heuristics to the proposition for supercritical BPREs with emigration were given in Remark \ref{remark:heuristic}. That the result should also hold for critical BPREs with emigration and specific subcritical BPREs with emigration, is motivated by recent work on BPREs, see e.g.\ \cite{BirknerGeigerKersting05, AfanasyevEtAl05, AfanasyevEtAl12} and references therein. There, it was shown that critical and so-called weakly subcritical BPREs behave in a supercritical manner when conditioned on survival. Thus, one can hope that the prize to pay for survival is negligible compared to the growth of an supercritical BPRE with emigration. 
\end{remark}

\begin{proof}[Proof of Proposition \ref{prop:supercriticalBPMRE}]
  The key idea of the proof is to couple $(Z_n)_{n\geq0}$ and
  a process $(X_n)_{n\geq0}$ that is similar to a random difference equation. More precisely, let $X_0:=1$ and recursively
    \[
      X_n:=(\rho_n X_{n-1}-M_n)_+
    \]
  for $n\in\N$. Note the analogy to the idea in the proofs of Theorem 2.2 in \cite{Bauernschubert12} and Theorem 4 in \cite{Bauernschubert13}.
  If the sequence $(M_n)_{n\geq1}$ is neglected, the growth of $(X_n)_{n\geq0}$ is determined by its ``associated random walk''.
  The same random walk basically describes the behavior of the BPRE without migration, see for instance \cite{BirknerGeigerKersting05}.
  Therefore, let us define $U_i:=\log \rho_i$ for $i\in\N$ and $Y_i:=U_1 +\ldots + U_i$ for $i\in\N$.
  Since $\mathbb{E}[\rho_1]=\mathbb{E}[\exp{U_1}]>1$ by assumption, we can find for every $0<\delta<1$ some $\kappa>0$ and $0<\tilde{\kappa}\leq 1$ such that
    \begin{equation}
      \label{eq:U1>kappa}
      \mathbb{P}[U_1>\kappa]^{\tilde{\kappa}} >1-\delta.
    \end{equation}

  Set $\epsilon:=\kappa\cdot\tilde{\kappa}$ and for $m\in\N$
  \begin{equation}
    \label{def:A_m}
     A(m):=\{\text{for all } 1\leq j\leq m: Y_j\geq \epsilon\cdot j\}.
  \end{equation}

  We control the probability of $A(m)$ from below by
    \begin{align}
      \label{eq:Prob_Am}
      \mathbb{P}[A(m)]
        & \geq \mathbb{P}[\forall 1\leq i\leq \lceil\tilde{\kappa}m\rceil: U_i\geq\kappa,\; \forall \lceil\tilde{\kappa}m\rceil <  j\leq m: Y_j-Y_{\lceil\tilde{\kappa}m\rceil}\geq0]\notag\\
        & \geq \mathbb{P}[U_1>\kappa]^{\tilde{\kappa}m+1} \cdot \mathbb{P}[\forall 1\leq j\leq (1-\tilde{\kappa})m: Y_j\geq0]\notag\\
        & \geq \mathbb{P}[U_1>\kappa]^{\tilde{\kappa}m+1} \cdot \mathbb{P}[\forall 1\leq j\leq m: Y_j\geq0]. 
    \end{align}
    
  In order to control the emigration define for $m$ and $k\in\N$, with $m\geq k$
    \[
      B(m,k):= \{M_1=\ldots=M_{k}=0, \forall  k<n\leq m: M_n<n\}.
    \]
  The events $A(m)$ and $B(m,k)$ are independent under $\mathbb{P}$ and
    \begin{align}
      \label{eq:Prob_Bm}
      \mathbb{P}[B(m,k)]=\mathbb{P}[M_1=0]^k\cdot\prod_{j=k+1}^{m}\mathbb{P}[M_1<j].
    \end{align}

  On $A(m)\cap B(m,k)$ with $m\geq k$, it is obtained that for all $1\leq j\leq k$
    \[
      X_j=\exp(Y_j)\geq \exp(\epsilon j)
    \]
  and by induction for all $k< j\leq m$
    \begin{align*}
      X_j 
      \geq\exp(Y_j)\left(1-\sum_{n=k+1}^{j}n \exp(-Y_n)\right)
 \geq\exp(Y_j)\left(1-\sum_{n=k+1}^{\infty}n \exp(-\epsilon n)\right).
    \end{align*}
  Thus, there exists $0<c_1<1$ such that for sufficiently large $k$ and all $m\geq k$, on $A(m)\cap B(m,k)$,
    \begin{equation}
      \label{eq:X_on_AmBm}
      X_j\geq c_1 \exp(Y_j)\geq c_1 \exp(\epsilon j)>0\quad\text{ for all } 1\leq j \leq m.
    \end{equation}

  Fix $k\in\N$ such that (\ref{eq:X_on_AmBm}) holds.
  For every $m\geq k$
    \begin{align}
      \label{eq:sum_Z_j1}
      E_0\Bigg[\sum_{j\geq 0} Z_j\Bigg]
      & \geq E_0\Bigg[\sum_{j\geq 0} Z_j, A(m),B(m,k),\forall j\leq m :Z_j\geq\frac{X_j}{j} \Bigg] \notag\\
      & \geq E_0\left[\frac{1}{m}X_m, A(m),B(m,k),\forall j\leq m :Z_j\geq\frac{X_j}{j} \right]\notag\\
      & \geq \frac{c_1}{m}E_0\left[\exp(Y_m), A(m),B(m,k),\forall j\leq m :Z_j\geq\frac{X_j}{j} \right].
  \end{align}

  For the moment let us have a closer look at $P_{(p,M)}[\forall j\leq m: Z_j\geq\frac{ X_j}{j}]$ on $A(m)\cap B(m,k)$.
  We can write
    \begin{align}
      \label{eq:totale_Wkeit}
      P_{(p,M)}&\left[\forall j\leq m: Z_j\geq\frac{ X_j}{j}\right]\notag\\
      & = P_{(p,M)}[Z_1\geq X_1]\prod_{j=2}^{m}P_{(p,M)}\left[Z_j\geq\frac{ X_j}{j}\Big|Z_{j-1}\geq\frac{X_{j-1}}{j-1} ,\ldots, Z_1\geq X_1\right] 
    \end{align}
  and obtain, with \eqref{eq:X_on_AmBm}, on $A(m)\cap B(m,k)$ for $2\leq j\leq m$
    \begin{align}
      \label{eq:totale_Wkeit_2}
       P&_{(p,M)}\left[Z_j\geq\frac{X_j}{j}\Big|Z_{j-1}\geq\frac{ X_{j-1}}{j-1},\ldots, Z_1\geq X_1\right]\notag\\
      & = \sum_{n\geq \frac{c_1\exp(Y_{j-1})}{j-1}}P_{(p,M)}\left[\sum_{i=1}^{n}\xi_{i}^{(j)}-M_j\geq \frac{\rho_j X_{j-1}-M_j}{j}\Big| Z_{j-1}=n\geq\frac{X_{j-1}}{j-1} ,\ldots, Z_1\geq X_1\right]\notag\\
      & \hphantom{=\sum_{n\geq \frac{c_1\exp(Y_{j-1})}{j-1}}}
         \cdot P_{(p,M)}\left[Z_{j-1}=n \Big|Z_{j-1}\geq\frac{X_{j-1}}{j-1} ,\ldots, Z_1\geq X_1\right]\notag\\
      & \geq \sum_{n\geq \frac{c_1\exp(Y_{j-1})}{j-1}}P_{(p,M)}\left[\sum_{i=1}^{n}\xi_{i}^{(j)}\geq \frac{\rho_j (j-1)n}{j} +(1-\frac{1}{j})M_j\right]\notag\\
      & \hphantom{\geq \sum_{n\geq \frac{c_1\exp(Y_{j-1})}{j-1}}}
      \cdot P_{(p,M)}\left[Z_{j-1}=n \Big|Z_{j-1}\geq\frac{X_{j-1}}{j-1} ,\ldots, Z_1\geq X_1\right].
    \end{align}    
  Further calculations yield for $(p,M)$ satisfying $A(m)\cap B(m,k)$, for $n\geq \frac{c_1\exp(Y_{j-1})}{j-1}$ and $2\leq j\leq m$
    \begin{align}
        \label{eq:totale_Wkeit_3}
       P_{(p,M)}\left[\sum_{i=1}^{n}\xi_{i}^{(j)}\geq \frac{\rho_j (j-1)n}{j} +(1-\frac{1}{j})M_j\right]
      & \geq P_{(p,M)}\left[\sum_{i=1}^{n}\xi_{i}^{(j)}\geq \frac{\rho_j (j-1)n}{j} +j\right]\notag\\
      & = P_{(p,M)}\left[n\rho_j - \sum_{i=1}^{n}\xi_{i}^{(j)}\leq \frac{n\rho_j}{j}-j\right].
    \end{align}
  Note that here $n\rho_j\geq \frac{c_1\exp(Y_{j})}{j-1}\geq \frac{c_1\exp(\epsilon j)}{j-1}$ by \eqref{eq:X_on_AmBm}. 
  Choose $j_0\in\N$ such that $c_1\exp(\epsilon j)j^{-3}\geq 2$ for all $j\geq j_0$.
  As in \cite[p.\ 643]{Bauernschubert12} we apply now Chebyshevs inequality. For sufficiently large $m$, we get for all $j_0 <j\leq m$
    \begin{align}
        \label{eq:totale_Wkeit_4}
      & P_{(p,M)}\left[n\rho_j - \sum_{i=1}^{n}\xi_{i}^{(j)}\leq \frac{n\rho_j}{j} -j\right]
       \geq P_{(p,M)}\left[\Big|n\rho_j - \sum_{i=1}^{n}\xi_{i}^{(j)}\Big|\leq \frac{n\rho_j}{j}-j\right]\notag\\
      & \geq 1- \frac{n \Var_{(p,M)}(\xi_{1}^{(j)})}{(\frac{n\rho_j}{j}-j)^2}
       =  1- \frac{n \Var_{(p,M)}(\xi_{1}^{(j)}) j^2}{(1-\frac{j^2}{n\rho_j})^2(n\rho_j)^2}
       \geq 1 - 4 \frac{j^2}{n\rho_j p_j}.
    \end{align}
  On $A(m)$ we have for all $n\geq \frac{c_1}{j-1}\exp(Y_{j-1})$ on the one hand $n\rho_j p_j\geq \frac{c_1}{j-1}\exp(\epsilon j)p_j$ and on the other hand $n\rho_j p_j=n(1-p_j)\geq \frac{c_1}{j-1}\exp(\epsilon (j-1))(1-p_j)$. Thus, $n\rho_j p_j\geq \frac{1}{2}\cdot\frac{c_1}{j-1}\exp(\epsilon (j-1))$.
  
  This gives together with (\ref{eq:totale_Wkeit_2}),  (\ref{eq:totale_Wkeit_3}) and (\ref{eq:totale_Wkeit_4}) for all $j_0\leq j \leq m$
  \begin{align}
    \label{eq:totale_Wkeit_5}
     P_{(p,M)}\left[Z_j\geq\frac{ X_j}{j} \Big|Z_{j-1}\geq\frac{ X_{j-1}}{j-1},\ldots, Z_1\geq X_1\right]
     \geq 1- c_2 j^3 e^{-\epsilon j}
  \end{align}
  for some $c_2>0$. Hence, there is some $j_1\geq j_0$ and some constant $c_3>0$ such that a similar calculation as in (\ref{eq:totale_Wkeit}) yields together with \eqref{eq:totale_Wkeit_5} 
  for all large $m$
  \begin{align*}
    P_{(p,M)}\left[\forall j\leq m: Z_j\geq\frac{X_j}{j} \right]
    & \geq P_{(p,M)}\left[\forall j\leq j_1: Z_j\geq\frac{X_j}{j}\right] \prod_{i\geq j_1}\left(1-c_2 i^3 e^{-\epsilon i}\right)\notag\\
    & \geq c_3 P_{(p,M)}\left[\forall j\leq j_1: Z_j\geq\frac{\exp(Y_j)}{j} \right].
  \end{align*}

  The last inequality holds since $X_i\leq\exp(Y_i)$ for all $i\in\N$ by definition.
  Recall (\ref{eq:sum_Z_j1}), \eqref{def:A_m}, (\ref{eq:Prob_Am}) and the independence of $(p,M)$. We obtain for sufficiently large $m$ (such that $\lceil\tilde{\kappa}m\rceil\geq j_1$) and some constant $c_4>0$ that
  \begin{align}
  \label{eq:Abschaetzung3}
     E_0\Bigg[\sum_{j\geq 0} Z_j\Bigg]&
      \geq \frac{c_4}{m} \mathbb{E}\left[\exp(Y_m) P_{(p,M)}\left[\forall j\leq j_1: Z_j\geq\frac{\exp(Y_j)}{j} \right],A(m), B(m,k)\right]\notag\\
    & \geq \frac{c_4}{m} \mathbb{E}\left[\exp(Y_{j_1})P_{(p,M)}\left[\forall j\leq j_1: Z_j\geq\frac{\exp(Y_j)}{j}\right] ,\forall i\leq j_1: U_i\geq\kappa, B(m,k)\right]\notag\\
    & \quad\cdot \mathbb{E}\Bigg[\prod_{i=j_1+1}^{m}\rho_i,\forall  j_1<i\leq\lceil\tilde{\kappa}m\rceil: U_i\geq\kappa, \forall \lceil\tilde{\kappa}m\rceil<j\leq m: Y_j-Y_{\lceil\tilde{\kappa}m\rceil}\geq0\Bigg].
  \end{align}  
  Since $\exp(Y_{j_1})P_{(p,M)}\left[\forall j\leq j_1: Z_j\geq\exp(Y_j) j^{-1}\right]$ and $\{\forall i\leq j_1: U_i\geq\kappa\}$ only depend on $(p_i,M_i)_{1\leq i\leq j_1}$ we obtain, by independence of $(p,M)$, for some constant $c_5>0$ 
    \[
      \mathbb{E}\left[\exp(Y_{j_1})P_{(p,M)}\left[\forall j\leq j_1: Z_j\geq\frac{\exp(Y_j)}{j}\right] ,\forall i\leq j_1: U_i\geq\kappa, B(m,k)\right]
      = c_5 \mathbb{P}[B(m,k)]
    \]
  for all large $m$.
The FKG inequality, see for instance \cite[Theorem (2.4), p.\ 34]{Grimmett_Perc99}, gives
    \begin{align*}
      \mathbb{E}
       &  \Bigg[\prod_{i=j_1+1}^{m}\rho_i,\forall  j_1<i\leq\lceil\tilde{\kappa}m\rceil: U_i\geq\kappa, \forall \lceil\tilde{\kappa}m\rceil<j\leq m: Y_j-Y_{\lceil\tilde{\kappa}m\rceil}\geq0\Bigg]\\
       & \geq \mathbb{E}\Bigg[\prod_{i=j_1+1}^{m}\rho_i \Bigg] \mathbb{P}\left[\forall  j_1<i\leq\lceil\tilde{\kappa}m\rceil: U_i\geq\kappa, \forall \lceil\tilde{\kappa}m\rceil<j\leq m: Y_j-Y_{\lceil\tilde{\kappa}m\rceil}\geq0\right]
    \end{align*}
  This inequality can be applied here, since $(p_j)_{j\in\N}$ is a sequence of $[0,1]$-valued i.i.d.\ random variables, $\mathbb{E}[\rho_1^2]<\infty$, and $\prod_{i=j_1+1}^{m}\rho_i$ and 
  $\mathbf{1}_{\{\forall i\leq\lceil\tilde{\kappa}m\rceil: U_i\geq\kappa, \forall \lceil\tilde{\kappa}m\rceil<j\leq m: Y_j-Y_{\lceil\tilde{\kappa}m\rceil}\geq0\}}$ are both monotonically decreasing functions in $(p_j)_{j\in\N}$ with respect to the usual partial order on $[0,1]^{\N}$.
  
  Hence, together with \eqref{eq:Abschaetzung3}, we have for some constant $c_6>0$ that
  \begin{align}
  \label{eq:Abschaetzung4}
     E_0\Bigg[\sum_{j\geq 0} Z_j\Bigg]\geq  \frac{c_6}{m} \mathbb{P}[B(m,k)] \mathbb{E}[\rho_1]^m (\mathbb{P}[U_1>\kappa]^{\tilde{\kappa}})^m \mathbb{P}[\forall 1\leq j\leq m: Y_j\geq0].
  \end{align}

  Thus, $E_0[\sum_{j\geq 0} Z_j]$ is infinite if the right hand side of (\ref{eq:Abschaetzung4}) goes to infinity for $m\to\infty$. Recall from (\ref{eq:Prob_Bm}) that $\mathbb{P}[B(m,k)]=\mathbb{P}[M_1=0]^k\prod_{j=k+1}^{m}\mathbb{P}[M_1<j]$ and remark that $\mathbb{P}[M_1<j]\to\mathbb{P}[M_1<\infty]$ as $j\to\infty$.
  Furthermore it is known in the case $\mathbb{E}[\log\rho_1]\geq0$ that $\mathbb{P}[\forall 1\leq j\leq m: Y_j\geq0]$ eventually exceeds $\frac{1}{\sqrt{m}}$ up to some multiplicative constant, see for instance \cite[XII.7]{Feller22}. 
  Thus the proposition follows immediately for supercritical or critical BPREs with emigration 
  when we choose $\delta$ in (\ref{eq:U1>kappa}) so small that $(1-\delta)\mathbb{P}[M_1<\infty]\mathbb{E}[\rho_1]>1$.
  
  Let $(Z_n)_{n\geq0}$ be a subcritical BPRE with emigration and $\mathbb{E}[\rho_1]>1$. 
   Due to (\ref{eq:Abschaetzung4}) the behavior of $\mathbb{P}[\forall 1\leq j\leq m: Y_j\geq0]$, as $m$ goes to infinity, is of interest.
  If the distribution of $U_1$ is non-lattice $\mathbb{P}[\forall 1\leq j\leq m: Y_j\geq0]$ is of order $m^{-\frac{3}{2}}\gamma^m$. (Recall that $\gamma=\mathbb{E}[\exp(\beta U_1)]<1$ with $\beta$ such that $\mathbb{E}[U_1\exp(\beta U_1)]=0$.)
  For references see for instance \cite[Theorem II]{Doney89} or \cite[Theorem 1.1 and Corollary 1.2]{AfanasyevEtAl12}.
  Thus, $E_0[\sum_{j\geq 0} Z_j]=\infty$ if $\gamma\mathbb{E}[\rho_1]\mathbb{P}[M_1<\infty]>1$ and the proposition is proven.
  For the lattice case, some monotonicity argument can be used.
\end{proof}

\section{Connection between random walks and branching processes}
\label{sec:translationERWRE_BPMRE}

  We turn now to the correspondence between ERWREs and certain BPMREs. Recall from the introduction that an RWRE is perturbed by cookies of maximal strength and that our aim is to study the speed of this ERWRE. In the current section we suppose that, additionally to Assumption \ref{as_A_paper3},
  the drift induced by the cookies wins, i.e.\ that
  \begin{equation}
    \label{eq:as_righttransient}
    P_0\left[\lim_{n\to\infty}S_n=+\infty\right]=1.
  \end{equation}
  Criteria for transience to the right are given in Theorem \ref{th:Bauernschubert12} and in Theorem~1 of \cite{Bauernschubert13} in the case of a left-transient or recurrent underlying RWRE. If the RWRE is right-transient then monotonicity with respect to the environment implies (\ref{eq:as_righttransient}), see \cite[Lemma 15]{Zerner05} (the condition $\omega(x,i)\geq\frac{1}{2}$ for all $x\in\Z$ and $i\in\N$ in \cite{Zerner05} is not necessary for the proof of Lemma~15).
  Due to Theorem 4.1 in \cite{KosyginaZerner12} the speed of the ERWRE exists on $\{S_n\to\infty\}$. The question is if there is a phase transition between zero and positive speed.

  A well-known tool to study the speed of an one-dimensional ERW is the so-called regeneration or renewal structure, see \cite[Section~6]{Zerner08} or \cite[Section~4]{KosyginaZerner12} and references therein.
  According to Lemma~4.5 in \cite{KosyginaZerner12} there are $P_0$-a.s.\ infinitely many random times $j$ on the event $\{S_n\to\infty\}$ with $S_m<S_j$ for all $m<j$ and $S_k\geq S_j$ for all $k\geq j$. The increasing enumeration of these renewal times is denoted by $(\tau_k)_{k\in\N}$.
  By \cite[Lemma~4.5]{KosyginaZerner12} and \eqref{eq:as_righttransient},
  we have that $(S_n)_{0\leq n\leq \tau_1}$, $(S_n-S_{\tau_k})_{\tau_k\leq n\leq \tau_{k+1}}$, $k\geq1$, are independent under $P_0$, $(S_n-S_{\tau_k})_{\tau_k\leq n\leq \tau_{k+1}}$, $k\geq1$, have the same distribution under $P_0$ and $E_0[S_{\tau_2}-S_{\tau_1}]<\infty$. Theorem 4.6 in \cite{KosyginaZerner12} gives, $P_0$-a.s.,
  \[
    \nu=\lim_{n\to\infty}\frac{S_n}{n}=\frac{E_0[S_{\tau_2}-S_{\tau_1}]}{E_0[\tau_2-\tau_1]}.
  \]
  Thus,
  \begin{equation}
    \label{eq:speed_correspondence1}
    \nu=0 \quad \text{ iff }\quad E_0[\tau_2-\tau_1]=\infty.
  \end{equation}

  The key to study the distribution of $\tau_2-\tau_1$ relies on the discussion of a branching process with migration in random environment that corresponds to the ERWRE. Compare this method to the one used for RWRE in \cite{KestenKozlovSpitzer75} and for ERW in \cite[Section~2]{Basdevant08b}, \cite[Section~6]{Zerner08} and \cite[Section~2]{KosyginaMountford11}, see also \cite[Section~5]{KosyginaZerner12} and references therein. For details concerning the connection we refer the reader to the specific sections in \cite{Zerner08, KosyginaMountford11, KosyginaZerner12}.

  Let us consider the so-called \emph{backward branching process} of the ERWRE. Therefore, recall that $(S_n)_{n\geq0}$ is transient to the right by (\ref{eq:as_righttransient}) and thus $\tau_1<\tau_2<\infty$ $P_0$-a.s.
  As in \cite[Section 6]{Zerner08}, denote by
  \[
    D_k:=\#\{n\in\N:\tau_1<n<\tau_2, S_n=S_{\tau_2}-k \text{ and } S_{n+1}=S_{\tau_2}-k-1 \}, \quad k\in\N_0,
  \]
  the number of downcrossings from $S_{\tau_2}-k$ to $S_{\tau_2}-k-1$ between times $\tau_1$ and $\tau_2$. The number of upcrossings in this time interval is $S_{\tau_2}-S_{\tau_1}+\sum_{k\geq0}D_k$. Hence
  \begin{align}
    \label{eq:connection1}
    \tau_2-\tau_1=S_{\tau_2}-S_{\tau_1}+2\sum_{k\geq0}D_k,
  \end{align}

  and thus $E_0[\tau_2-\tau_1]=\infty$ if and only if $E_0[\sum_{k\geq0}D_k]=\infty$.

  It can be shown like in the proof of \cite[Lemma 12]{Zerner08} that $(D_k)_{k\geq0}$ is distributed, under $P_0$, like a BPMRE $(W_k)_{k\geq0}$ defined by $W_0=0$ and 
  \[
    W_k=\mathbf{1}_{\{k\leq T_0^{W}\}} \sum_{i=1}^{W_{k-1}+1-M_k}\xi_{i}^{(k)},
  \]
  where $\xi_i^{(j)}$, $i,j\in\N_0$, are random variables on $\Omega'$ that are independent under $P_{(p,M)}$, and $P_{(p,M)}[\xi_i^{(j)}=n]=(1-p_j)^n p_j$ for $n\in\N_0$. The random variable $T_0^{W}:=\inf\{k\geq1: W_k=0\}$ denotes the first time of extinction of $(W_k)_{k\geq0}$.
  
  The correspondence now yields by \eqref{eq:connection1}
  \begin{equation*}
    E_0[\tau_2-\tau_1]=\infty \quad \text{ iff }\quad  E_0\Bigg[\sum_{k\geq0}W_k\Bigg]=E_0\Bigg[\sum_{k=1}^{T_0^{W}-1}W_k\Bigg]=\infty
  \end{equation*}
  and therefore by (\ref{eq:speed_correspondence1})
  \begin{equation}
    \label{eq:speed_correspondence2}
    \nu=0 \quad \text{ iff } \quad E_0\Bigg[\sum_{k=1}^{T_0^{W}-1}W_k\Bigg]=\infty.
  \end{equation}
\label{sec:translationERWRE_BPMRE_end}  
  
\section{On the speed of the random walk, proofs}
\label{sec:proof_of_speed}

  At first we show that $(S_n)_{n\geq0}$ satisfies a strong law of large numbers.
  
  \begin{theorem}
    \label{th:speed_exists}
    Let Assumption \ref{as_A_paper3} hold. Then
    there exists a non-random $\nu\in[-1,1]$ such that $\lim_{n\to\infty}S_n/n=\nu$ $P_0$-a.s. 
  \end{theorem}
  
  \begin{proof}
    If $\mathbb{E}[(\log M_0)_+]<\infty$ and $\mathbb{E}[\log\rho_0]>0$, the ERWRE goes to $-\infty$ $P_0$-a.s.\ by Theorem \ref{th:Bauernschubert12}(i). Then, $(S_n)_{n\geq0}$ satisfies a strong law of large numbers according to \cite[Theorem 4.1]{KosyginaZerner12}.
    
    If $\mathbb{E}[(\log M_0)_+]=\infty$ and $\mathbb{E}[\log\rho_0]> 0$,
    \cite[Proposition 4.1]{Bauernschubert12} and monotonicity with respect to the environment --- see \cite[Lemma 15]{Zerner05} which also holds for $\Omega=([0,1]^{\N})^{\Z}$ --- imply
    \[
     P_0\left[\sup_{n\geq0} S_n=\infty\right]=1.
    \]
  The same holds if the underlying RWRE is recurrent or right-transient, i.e.\ $\mathbb{E}[\log\rho_0] \leq0$.
  Since $\mathbb{P}[M_0=0]>0$ a weak ellipticity condition as described in \cite[p.\ 108]{KosyginaZerner12} holds for the environment $\omega$. Theorem 3.2 in \cite{KosyginaZerner12} --- weak ellipticity is sufficient for case (d) in the proof --- yields $P_0[|S_n|\to\infty]\in\{0,1\}$.
  Since $P_0[\liminf_{n\to\infty} S_n \in\{\pm \infty\}]=1$ by \cite[Lemma 2.2]{KosyginaZerner12} we get
    \[
      P_0\left[\inf_{n\geq0} S_n = -\infty\right]\in\{0,1\}.
    \]
  Thus, $(S_n)_{n\geq0}$ satisfies a strong law of large numbers according to \cite[Theorem 4.1]{KosyginaZerner12}.
  \end{proof}

  We now show Theorems \ref{th:1_transientRE} and \ref{th:2_transientRE}. Since their proofs are overlapping concerning the applied tools, we will merge them and structure it along the different methods.
  
  Let us first introduce some more notation. Recall the notation $P_0$, $E_0$, $P_{(p,M)}$ and $E_{(p,M)}$ as defined in the introduction.
  The quenched measure in an environment without cookies will be denoted by $P_{p}:=P_{(p_x,0)_{x\in\Z}}$ and the corresponding annealed measure by $P_{RE}[\cdot]:=\mathbb{E}[P_{p}[\cdot]]$. Note that under this measure $(S_n)_{n\geq0}$ is known as RWRE with start in $0$. Furthermore, situations will be considered, where there are only cookies on sites less or equal to zero. For this setting we write $P_{p,\leq 0}:=P_{(p_x,M_x)_{x\in-\N_0},(p_x,0)_{x\in\N}}$ and $P_{RE,\leq 0}[\cdot]:=\mathbb{E}[P_{p,\leq 0}[\cdot]]$. The corresponding expectations are $E_{p}$, $E_{RE}$, $E_{p,\leq 0}$ and $E_{RE,\leq 0}$ respectively. The speed or limit in the law of large numbers under $P_{RE}$, if it exists, is denoted by $\nu_{RE}$.
  For $k\in\Z$ let $T_k:=\inf\{n\geq0: S_n=k\}$ be the first hitting time of $k$.

\begin{proof}[Proof of Theorems \ref{th:1_transientRE} and \ref{th:2_transientRE}]
\hspace{\fill}

\textit{First part:} We shall prove Theorems \ref{th:1_transientRE}(iii) and \ref{th:2_transientRE}(ii)-(iii) by using the results on RWRE (Theorem \ref{th:RWRE_speed}) and monotonicity of the speed with respect to the environment, \cite[Proposition 4.2]{KosyginaZerner12}.
  
  First, let $M_0$ be $\{0,\infty\}$-valued, $\mathbb{P}$-a.s. Thus, we consider a.s.\ environments where the random walker encounters
  infinite cookie stacks and between those stacks an environment known from RWRE. The cookie piles can be regarded as ``one-way doors'': the random walker goes through from the left to the right but has no chance to get back.  
  Note that in this setting the model can be interpreted in terms of an RWRE in the following way.
  Define $\tilde{p}_x:=1$ if $M_x=\infty$ and $\tilde{p}_x:=p_x$ otherwise for $x\in\Z$.
  Then for $\mathbb{P}$-a.e.\ environment $(p,M)$, $(S_n)_{n\geq0}$ has the same distribution under $P_{(p,M)}$ and $P_{\tilde{p}}$ where $\tilde{p}:=(\tilde{p}_x)_{x\in\Z}$.
  Note that, under Assumption \ref{as_A_paper3}, $(\tilde{p}_x)_{x\in\Z}$ is i.i.d.\ under $\mathbb{P}$, $\tilde{p}_x\in (0,1]$, $\mathbb{P}[\tilde{p}_0=1]=\mathbb{P}[M_0=\infty]$ and $\mathbb{E}[\log \tilde{\rho}_0]$ is well defined (with possible value $-\infty$) where $\tilde{\rho}_x:=\frac{1-\tilde{p}_x}{\tilde{p}_x}$.
  
  According to Theorem \ref{th:RWRE_speed}, we obtain for the speed $\nu>0$ if $\mathbb{E}[\tilde{\rho}_0]<1$ and
  $\nu= 0$ if $\mathbb{E}[\tilde{\rho}_0]\geq 1$.
  Now, the statements in Theorems \ref{th:1_transientRE}(iii), \ref{th:2_transientRE}(iii) and \ref{th:2_transientRE}(ii)
  follow for the case where $M_0$ is $\{0,\infty\}$-valued, since 
  \[
   \mathbb{E}[\tilde{\rho}_0]=\mathbb{E}\left[\frac{1-\tilde{p}_0}{\tilde{p_0}}, \tilde{p}_0<1\right]=\mathbb{E}[\rho_0] \mathbb{P}[M_0<\infty].
  \]
  
  Consider the general case where $M_0$ is $\N_0\cup\{\infty\}$-valued.
  If $\mathbb{E}[\rho_0]\mathbb{P}[M_0=0]\geq 1$ then we replace finite (but not empty) cookie piles by infinite ones. Applying the monotonicity of the speed yields $\nu= 0$ and thus Theorem \ref{th:2_transientRE}(ii) follows.
  If $\mathbb{E}[\rho_0]\mathbb{P}[M_0<\infty]<1$ we do the reverse: Let all finite cookie stacks vanish and obtain thus an environment with infinitely many cookies or none per integer. Exploiting again monotonicity gives $\nu>0$ 
  and hence statement (iii) in Theorems \ref{th:1_transientRE} and \ref{th:2_transientRE}.
  
\textit{Second part:} We shall prove Theorem \ref{th:1_transientRE}(i)-(ii) under the additional assumptions $\mathbb{E}[\log\rho_0]>0$ and $\mathbb{E}[(\log M_0)_+]<\infty$, i.e.\ the underlying RWRE goes to $-\infty$ $P_0$-a.s.\ by Theorem \ref{th:Bauernschubert12}(i). For this setting we use the formula for the speed of an ERW in \cite[Proposition 13]{Zerner08}, see also \cite[Theorem 13]{Zerner05}. 
 
  Let $\mathbb{E}[\log\rho_0]>0$, $\mathbb{E}[(\log M_0)_+]<\infty$ and $\mathbb{E}[M_0]=\infty$. 
  According to \cite[Theorem 4.1]{KosyginaZerner12} and \cite[Proposition 13]{Zerner08} and their proofs, $(S_n)_{n\geq0}$ satisfies a law of large numbers 
  with 
  \[
   \frac{1}{\nu}=-\sum_{j\geq1}P_0[T_{-j-1}-T_{-j}\geq j]\leq -\sum_{j\geq1}P_0[M_{-j}\geq j]=-\mathbb{E}[M_0]=-\infty.
  \]
  
  Hence the speed $\nu$ is zero.
  
  Let $\mathbb{E}[M_0]<\infty$ and $\mathbb{E}[\log\rho_0]>0$. If $\mathbb{E}[\rho_0^{-1}]\geq1$ then $\nu_{RE}=0$ by Theorem \ref{th:RWRE_speed}.
  Using the monotonicity of $\nu$ with respect to the environment, see \cite[Proposition 4.2]{KosyginaZerner12}, and the fact that
  $S_n\to -\infty$ $P_0$-a.s.\ we obtain
    \[
     0=\nu_{RE}\leq \nu\leq0.
    \]
  If $\mathbb{E}[\rho_0^{-1}]<1$ then $E_{RE}[T_{-1}]<\infty$, compare to \cite[Theorem (1.15)]{Solomon75} or see
  \cite[Lemma 2.1.12, Theorem 2.1.9 + Remark]{Zeitouni}. Furthermore
  \begin{align*}
    &\sum_{j\geq1}P_0[T_{-j-1}-T_{-j}\geq j]=\sum_{j\geq1}P_{RE,\leq0}[T_{-j-1}-T_{-j}\geq j]\\
    &=\sum_{j\geq1}P_{RE,\leq0}[T_{-1}\geq j]=E_{RE,\leq0}[T_{-1}].
  \end{align*}
  
  For the first equality note that $P_{(p,M)}[T_{-j-1}-T_{-j}\geq j]$ is a function of $(p_x)_{-j\leq x\leq0}$ and $M_{-j}$ since $(S_n)_{n\geq0}$ is a nearest neighbor random walk and by time 
  $T_{-j}$ --- which is finite $P_0$-a.s.\ --- all cookies on integers $-j+1\leq x\leq0$ are eaten by the walker.
  The second equality holds by the strong Markov property and shift-invariance under $\mathbb{P}$.
  Now, it is obtained by independence of $M_0$ and $p$ that
  \begin{align*}
    E_{RE,\leq0}[T_{-1}]
    & =\mathbb{E}[M_0 (1+E_{1,p}[T_{0}])+E_{p}[T_{-1}]]\\
    & =\mathbb{E}[M_0](1+ E_{RE}[T_{-1}])+E_{RE}[T_{-1}]<\infty,
  \end{align*}
  where $E_{1,p}[T_{0}]$ is the expected time of a random walker in environment $p$ and start in $1$ to reach $0$. 
  Thus Theorem \ref{th:1_transientRE}(i) is proven completely and Theorem \ref{th:1_transientRE}(ii) under the additional assumption $\mathbb{E}[\log\rho_0]>0$.
  
  \textit{Third part:} Finally we prove Theorem \ref{th:1_transientRE}(ii) under the additional assumption that $\mathbb{E}[(\log M_0)_+]=\infty$ if $\mathbb{E}[\log\rho_0]> 0$, and Theorem \ref{th:2_transientRE}(i).
  As can be seen from the proof of Theorem \ref{th:speed_exists}, $(S_n)_{n\geq0}$ is either $P_0$-a.s.\ recurrent --- and then Theorem \ref{th:speed_exists} yields $\nu=0$ --- or $P_0$-a.s.\ transient to the right. 
  It remains to consider the latter case for which we will
  use the link between ERWREs and BPMREs from Section \ref{sec:translationERWRE_BPMRE}.
  
  Let $S_n\to+\infty$ $P_0$-a.s.\ and recall the notation from Section \ref{sec:translationERWRE_BPMRE}.
  We will show that 
    \begin{equation}
    \label{eq:sum_W_unendlich}
      E_0\Bigg[\sum_{k=1}^{T_0^{W}-1}W_k\Bigg]=\infty.
    \end{equation}
  In order to obtain this result, $(W_n)_{n\geq0}$ is compared to the slightly different BPRE with emigration, that was introduced in Section \ref{sec:supercriticalBPMRE}, where no immigration occurs 
  and the number of emigrants is immediately subtracted from the population size. The single immigrant that appeared in the link between the ERWRE and the BPMRE in Section \ref{sec:translationERWRE_BPMRE} is neglected.
  For comparing the two BPMRE models, the sequence $(M_k)_{k\in\N}$ in the process in Section~\ref{sec:supercriticalBPMRE} has to be shifted. So let here $Z_0:=1$ and
    \[
      Z_k:=\left(\sum_{i=1}^{Z_{k-1}}\xi_{i}^{(k)}-M_{k+1}\right)_+, \quad k\geq1.
    \]
  Assume that $M_1=0$. Then, induction shows that
    \[
      (W_k-M_{k+1}+1)_+\geq Z_k \quad\text{ for all }0\leq k\leq T_0^W,  
    \]
  and hence for $1\leq k\leq T_0^W$
    \[
      Z_k=\left(\sum_{i=1}^{Z_{k-1}}\xi_{i}^{(k)}-M_{k+1}\right)_+\leq\sum_{i=1}^{Z_{k-1}}\xi_{i}^{(k)}\leq \sum_{i=1}^{(W_{k-1}-M_{k}+1)_+}\xi_{i}^{(k)}=W_k.
    \]
  In particular $T_0^Z:=\inf\{n\geq1: Z_n=0\}\leq T_0^W$. Thus,
    \[
      \sum_{k=1}^{T_0^{W}-1}W_k\geq \mathbf{1}_{\{M_1=0\}}\sum_{k=1}^{T_0^{Z}-1}Z_k=\mathbf{1}_{\{M_1=0\}}\sum_{k\geq1}Z_k.
    \]

  Note that $M_1$ and $(Z_k)_{k\geq0}$ are independent and $\mathbb{P}[M_1=0]>0$. 
  Since it is assumed in Theorem \ref{th:1_transientRE}(ii) that $\mathbb{E}[\rho_1]> \mathbb{P}[M_1<\infty]^{-1}$, and respectively in Theorem \ref{th:2_transientRE}(i) that $\mathbb{E}[\rho_1]\gamma> \mathbb{P}[M_1<\infty]^{-1}$, Proposition \ref{prop:supercriticalBPMRE} yields (\ref{eq:sum_W_unendlich}). So, the speed of the ERWRE is zero by (\ref{eq:speed_correspondence2}). 
\end{proof}

\begin{remark}
  Under Assumption \ref{as_A_paper3},
  the speed of the ERWRE is zero by Theorem \ref{th:2_transientRE}(ii) if $M_0\in\{0,\infty\}$ $\mathbb{P}$-a.s.\ and $\mathbb{E}[\rho_0]\geq\mathbb{P}[M_0<\infty]^{-1}$.
  The same result can be obtained if $M_0\in\{0,1,\infty\}$ $\mathbb{P}$-a.s.\ instead of $M_0\in\{0,\infty\}$.
  Let us give shortly the main ideas for a proof in this specific case. A regeneration structure similar to the one in Section \ref{sec:translationERWRE_BPMRE} can be established where the renewals are the first hitting times of the infinite cookie stacks. In order to obtain that the expected sojourn time of the random walker between two infinite cookie piles is infinite, decomposition techniques and calculations similar to those for RWREs in \cite[proof of Lemma~2.1.12]{Zeitouni} can be used.
  \end{remark}

\vspace{0.4cm}
\subsection*{Acknowledgements}
This work was funded by the ERC Starting Grant 208417-NCIRW and is included in my dissertation at Eberhard Karls Universit\"at T\"ubingen.
I am grateful to my advisor Prof.\ Martin P.\ W.\ Zerner for his continuous support and helpful suggestions.
I thank Elmar Teufl for many fruitful discussions.

\bibliographystyle{amsplain}
\providecommand{\bysame}{\leavevmode\hbox to3em{\hrulefill}\thinspace}
\providecommand{\MR}{\relax\ifhmode\unskip\space\fi MR }

\parindent=0pt 
\end{document}